\documentclass{lmcs}
\pdfoutput=1

\usepackage[utf8]{inputenc}

\usepackage{lastpage}
\lmcsdoi{22}{4}{22}
\lmcsheading{}{\pageref{LastPage}}{}{}%
{Jan.~23,~2026}{Nov.~24,~2026}{}


\usepackage{amsmath,amssymb}
\usepackage{tikz-cd}

\theoremstyle{plain}

\theoremstyle{definition}

\newcommand{\pbcorner}{\ar[dr, phantom, very near start, "\lrcorner"]}

\newcommand{\inl}{\mathtt{inl}}
\newcommand{\inr}{\mathtt{inr}}

\newcommand{\NN}{\mathbb{N}}

\newcommand{\QQ}{\mathbb{Q}}
\newcommand{\RR}{\mathbb{R}}
\newcommand{\II}{\mathbb{I}}

\newcommand{\yoneda}{\mathbf{y}}

\newcommand{\cubset}{\widehat{\square}}
\newcommand{\sets}{\mathbf{Set}}

\newcommand{\hprop}{\mathbf{hProp}}
\newcommand{\forallq}[2]{\forall #1 \mathpunct{.} #2}
\newcommand{\existsq}[2]{\exists #1 \mathpunct{.} #2}

\begin{document}

\title{Double negation stable h-propositions in cubical sets}

\author[A.~W.~Swan]{Andrew W Swan\lmcsorcid{0000-0002-7190-4870}}

\address{Faculty of Mathematics and Physics, University of Ljubljana, Slovenia}
\email{wakelin.swan@gmail.com}

\thanks{This material is based upon work supported by the Air Force Office of
Scientific Research under award number FA9550-21-1-0009. Any opinions,
findings, and conclusions or recommendations expressed in this
material are those of the author(s) and do not necessarily reflect the
views of the United States Air Force.}

\begin{abstract}
  We give a construction of classifiers for double negation stable h-propositions in a variety of cubical set models of homotopy type theory and cubical type theory. This is used to give some relative consistency results: classifiers for double negation stable propositions exist in cubical sets whenever they exist in the metatheory; the Dedekind real numbers can be added to homotopy type theory without changing the consistency strength; we construct a model of homotopy type theory with extended Church's thesis, which states that all partial functions with double negation stable domain are computable.
\end{abstract}
\maketitle

\section{Introduction}
\label{sec:introduction}

In approaches to the semantics of homotopy type theory (HoTT) based on model structures, such as simplicial sets \cite{voevodskykapulkinlumsdainess} and cubical sets \cite{bchcubicalsets, coquandcubicaltt, abchfl, awodey19}, we have categories that can be regarded as models of type theory in two different ways. Simplicial sets and cubical sets are toposes and as such can be viewed as models for extensional type theory. They also have notions of Kan fibration and homotopy that are combined with the topos structure to produce models of HoTT. In particular there are two different definitions of families of propositions in a given context, depending on whether we interpret equality according to the locally cartesian closed structure, or using homotopy as in \cite{awodeywarren}. A fibration between fibrant objects, \(f : X \to Y\), is a proposition in context \(Y\) according to the underlying topos structure when it is a monomorphism, which equivalently says the diagonal map \(X \to X \times_Y X\) has a section. It is a proposition from the perspective of homotopy when the map \(\operatorname{Path}_Y(X) \to X \times_Y X\) has a section, where \(\operatorname{Path}_Y(X)\) is the type of paths in \(X\) over \(Y\), defined as the following pullback.
\begin{displaymath}
  \begin{tikzcd}
    \operatorname{Path}_Y(X) \ar[d] \ar[r] & X^\II \ar[d] \\
    Y \ar[r] & Y^\II
  \end{tikzcd}
\end{displaymath}
To avoid confusion with monomorphisms, we will refer to the latter as \emph{homotopy propositions} or just \emph{h-propositions}.

In simplicial sets in a classical setting, where we have the law of excluded middle and axiom of choice, there is a tight correspondence between the two definitions. Every h-proposition is equivalent to a pullback of the coproduct inclusion \(1 \to 1 + 1\) and in particular must be equivalent to a monomorphism  (see e.g. \cite{kapulkinlumsdainelem} or \cite[Section 5.2]{christensennal}). On the other hand, in a constructive setting h-propositions can in general behave quite differently to monomorphisms (see e.g. \cite[Section 5]{uemuracubasm} or \cite[Section 4.1]{uemuraswan}). In this paper we will consider a restricted class of h-propositions: those which are stable under double negation, which we refer to as \emph{\(\neg \neg\)-stable h-propositions}. For this restricted class we can recover some of the good behaviour of h-propositions in a classical setting, and in particular construct classifying objects giving us a restricted form of resizing. We will show \(\neg \neg\)-stable h-propositions suffice to construct the Dedekind real numbers and to formulate and prove consistency of an extended version of Church's thesis for partial functions \(\NN \rightharpoondown \NN\).

\section{Review of cubical sets}
\label{sec:review-cubical-sets}

We recall some basic definitions and theorems about cubical sets. Although there are a few variations on the definition of cubical sets \cite{coquandcubicaltt, abchfl, awodey19}, for this paper we will only use a few properties that hold for several of the different definitions. Throughout we assume we are working in a metatheory of extensional type theory with propositional truncation, i.e. the internal language of a regular locally cartesian closed category \cite{awodeybauerpat, maiettimodular}.

Firstly, we assume that we are given a category \(\square\) which is a Lawvere theory, i.e. \(\square\) has finite products, and an object \(I\) such that any object is an \(n\)-fold product of \(I\) for some \(n \in \NN\). We write \(I^n\) as \([n]\). In particular \([0]\) is the terminal object of \(\square\). We refer to \(\square\) as the \emph{category of cubes}.

We refer to the presheaf category \(\widehat{\square}\) as the category of \emph{cubical sets}. For a cubical set, \(X\), we write the set at an object \([n]\) as \(X_n\), and for \(s : [n] \to [m]\) we write \(X_s\) for the function \(X_m \to X_n\).

Cubical sets can be used to model homotopy type theory. Maps \(X \to Y\) in \(\widehat{\square}\) may possess a kind of structure known as \emph{Kan fibration structure}, and types in homotopy type theory are interpreted as maps in \(\widehat{\square}\) together with Kan fibration structure. We refer to a pair consisting of a map \(f : X \to Y\) and a Kan fibration structure on \(f\) as a \emph{Kan fibration}, or just \emph{fibration}. A \emph{fibrant object} is an object \(X\) together with a Kan fibration \(X \to 1\).

\begin{defi}
  We say a map \(m : A \to B\) is a \emph{trivial cofibration} if we can choose a diagonal filler for each lifting problem of \(m\) against a fibration \(f : X \to Y\). That is, given a commutative square as in the solid lines in the diagram below, we have a choice of map \(j : B \to X\) making two commutative triangles, as in the dotted line below.
  \begin{displaymath}
    \begin{tikzcd}
      A \ar[r] \ar[d, "m"] & X \ar[d, "f"] \\
      B \ar[r] \ar[ur, dotted, "j" description] & Y
    \end{tikzcd}
  \end{displaymath}
\end{defi}

\noindent
We will use the following facts about Kan fibrations and the interpretation of homotopy type theory in cubical sets:
\begin{enumerate}
\item Contexts are interpreted as objects \(Y\), types in context \(Y\) are interpreted as Kan fibrations \(X \to Y\) and terms of a type \(f : X \to Y\)  are interpreted as sections of \(f\). The empty context is interpreted as the terminal object.
\item Fibrations are preserved by pullback along any map, and this is used to interpret substitution in HoTT.
\item The natural number object and initial object are fibrant and they are the underlying objects of the interpretations of the natural number type and empty type in the interpretation of HoTT.
\item Kan fibrations are closed under composition and composition is used to interpret \(\Sigma\) types in the model of HoTT.
\item Kan fibrations are closed under dependent product and dependent products are used to interpret \(\Pi\) types in the model of HoTT.
\item \(I\) has at least one global section, say \(\delta : 1 \to I\).
\item \label{it:kanishurewicz} For every object \(A\) and every map \(d : 1 \to \yoneda I\) the map \(d \times A : A \to \yoneda I \times A\) is a trivial cofibration. ``Kan fibrations are in particular Hurewicz fibrations.''
\item Cubical sets admit propositional truncation operators \(\| - \|_Y : \cubset/Y \to \cubset/Y\) which are preserved by reindexing and used to interpret propositional truncation in HoTT. (See e.g. \cite{cavalloharper, chmhitsctt, abchfl})
\item Any map between constant cubical sets is a fibration.
\end{enumerate}

\begin{prop}
  \label{prop:cubeshavepoints}
  Every object of \(\square\) admits a global section.
\end{prop}

\begin{proof}
  We show this for all objects \(n\) of \(\square\) by induction on \(n\). For \(n = 0\) \([n]\) is the terminal object, so we can use the identity map. For any \(n\), \([n + 1] \cong [n] \times [1]\), and so we have e.g. \([n] \times \delta : [n] \to [n + 1]\). Given a map \(s : [0] \to [n]\) we can compose to get \(([n] \times \delta) \circ s : [0] \to [n + 1]\).
\end{proof}

\begin{lem}
  \label{lem:termrepristcof}
  For any representable \(\yoneda [n]\), any map \(1 \to \yoneda [n]\) is a trivial cofibration.
\end{lem}

\begin{proof}
  Note that \(1 = \yoneda [0]\). Hence it suffices to show by induction on \(n\) that for any map \(s : [0] \to [n]\), \(\yoneda s : \yoneda [0] \to \yoneda [n]\) is a trivial cofibration. For \(n = 0\), since \([0]\) is the terminal object, any map \([0] \to [0]\) is an isomorphism, and so a trivial cofibration.
  
  Given any map \(s : [0] \to [n + 1]\), we have \([n + 1] \cong [n] \times [1]\), and so we can factor \(s\) as \(([n] \times d) \circ s'\), where \(d : [0] \to [1]\) is defined by \(d := \pi_1 \circ s\) and \(s' : [0] \to [n]\) is defined by \(\pi_0 \circ s\). By induction \(\yoneda s'\) is a trivial cofibration, and \(\yoneda [n] \times \yoneda d\) is a trivial cofibration, by \ref{it:kanishurewicz} in the list of basic facts about cubical sets above.
\end{proof}

\noindent
Since \(\square\) has a terminal object, the unique functor \(\square \to 1\) has a right adjoint. This induces a string of adjunctions \(\Delta \dashv \Gamma \dashv \nabla\) by reindexing and right Kan extension as illustrated below.
\begin{displaymath}
  \begin{tikzcd}[sep=5em]
    \sets \ar[r, bend left, yshift=0.6em, "\Delta" description] \ar[r, bend right, swap, yshift=-0.6em, "\nabla" description]
    \ar[r, phantom, yshift = 1em, "\perp"] \ar[r, phantom, yshift = -1em, "\perp"] & \cubset \ar[l, "\Gamma" description]
  \end{tikzcd}
\end{displaymath}

\noindent
Since \(\Delta\) can be explicitly described as reindexing along the unique functor \(\square \to 1\), one can easily show the following proposition.
\begin{prop}
  \(\Delta\) preserves all limits, colimits and dependent products.
\end{prop}

\begin{rem}
  The above results also hold for simplicial sets, so the arguments below will also apply there.
\end{rem}

\noindent
We can now show a key theorem about the behaviour of propositional truncation in cubical sets, using a technique due to Uemura. The idea here is that propositional truncation does not identify points of a type, but only adds new paths between them. Hence we should visualise h-propositions in general not as spaces with at most one point, but rather as many points where any two are connected by a path.
\begin{thm}
  \label{thm:detruncatedisccod}
  Suppose we are given a fibration \(f : W \to \Delta Z\) in cubical sets. If \(\| W \|_{\Delta Z} \to \Delta Z\) has a section, then so does the map \(\Gamma W \to \Gamma \Delta Z \cong Z\).
\end{thm}

\begin{proof}
  We first recall that \(\nabla_Z \Gamma W\) is defined as follows. The unit map \(Z \to \Gamma \Delta Z\) is an isomorphism, and so has an inverse, say \(i : \Gamma \Delta Z \stackrel{\cong}{\to} Z\). Under the adjunction \(\Gamma \dashv \nabla\), \(i\) corresponds to a map \(\Delta Z \to \nabla Z\), and composing \(i\) with \(\Gamma f : \Gamma W \to \Gamma \Delta Z\) gives a map \(\Gamma W \to Z\). We then define \(\nabla_Z \Gamma W\) as the pullback below.
  \begin{displaymath}
    \begin{tikzcd}
      \nabla_Z \Gamma W \ar[r] \ar[d] \pbcorner & \nabla \Gamma W \ar[d] \\
      \Delta Z \ar[r] & \nabla Z
    \end{tikzcd}
  \end{displaymath}
  The left hand map \(\nabla_Z \Gamma W \to \Delta Z\) is an h-proposition and in particular a Kan fibration - see \cite[Section 4.5]{uemuracubasm} for details. We can then extend the diagram above to the solid lines in the diagram below.
  \begin{displaymath}
    \begin{tikzcd}
      W \ar[d] \ar[drr, bend left] \ar[dr] & & \\
      \| W \|_{\Delta Z} \ar[dr] \ar[r, dotted] & \nabla_Z \Gamma W \ar[d] \ar[r] \pbcorner & \nabla \Gamma W \ar[d] \\
      & \Delta Z \ar[r] & \nabla Z
    \end{tikzcd}
  \end{displaymath}
  Since the map \(\nabla_Z \Gamma W \to \Delta Z\) is an h-proposition we obtain the dotted map in the diagram above. Composing this with the section of \(\| W \|_{\Delta Z} \to \Delta Z\) gives us a section of the map \(\nabla_Z \Gamma W \to \Delta Z\). We can then use this to obtain a section of \(\Gamma W \to \Gamma \Delta Z\) using the adjunction \(\Gamma \dashv \nabla\), as in \cite[Section 4.5]{uemuracubasm}.
\end{proof}

\begin{rem}
  Theorem~\ref{thm:detruncatedisccod} is fairly specific to cubical sets.
  In order to apply Uemura's proof we additionally need to assume that cofibrations are pointwise decidable and that the interval is representable, has disjoint endpoints and no points other than the endpoints.
  Most critically, Uemura's proof is specific to the definition of Kan fibration for cubical sets.
  This is the same as one of the common ways to define Kan fibration in simplicial sets, but e.g. does not apply to localisations, which can have a smaller class of fibrations on the same underlying category.
\end{rem}

\begin{cor}
  \label{cor:detruncate}
  Given a Kan fibration \(X \to Y\) in cubical sets, if the fibration \(\| X \|_Y \to Y\) has a section, then so does the map \(\Gamma X \to \Gamma Y\).
\end{cor}

\begin{proof}
  Write \(e : \Delta \Gamma Y \to Y\) for the counit map of the adjunction \(\Delta \dashv \Gamma\). We pullback the section of \(\| X \|_Y \to Y\) along \(e\) as illustrated below.
  \begin{displaymath}
    \begin{tikzcd}
      e^\ast \| X \|_{Y} \ar[r] \ar[d] \pbcorner & \| X \|_Y \ar[d] \\
      \Delta \Gamma Y \ar[r] \ar[u, bend left] & Y \ar[u, bend right]
    \end{tikzcd}
  \end{displaymath}
  Propositional truncation is stable under pullback, and so we have \(e^\ast \|X\|_Y \cong \| e^\ast(X)\|_{\Delta\Gamma Y}\). However, we can now apply Theorem \ref{thm:detruncatedisccod} with \(Z := \Gamma Y\) and \(W := e^\ast(X)\) and observe \(\Gamma (e^\ast(X)) \cong \Gamma X\) to obtain the conclusion.
\end{proof}

\noindent
Theorem~\ref{thm:detruncatedisccod} gives us an easy direct way to see how it can happen that not every h-proposition is equivalent to a monomorphism.
Suppose we are working internally in a category where the axiom of choice fails, i.e. where there is a regular epimorphism \(f : X \to Z\) in \(\sets\) that does not have a section.
Suppose that the h-proposition \(\| \Delta X \|_{\Delta Z} \to \Delta Z\) is logically equivalent to a monomorphism \(m : Y \to \Delta Z\).
Using the logical equivalence, the regular epimorphism \(\Delta(f) : \Delta X \to \Delta Z\) factors through \(m\), which implies \(m\) is an isomorphism.
Using the other direction of the logical equivalence, we see that \(\| \Delta X \|_{\Delta Z} \to \Delta Z\) has a section, and so \(f\) must also have a section.

\section{Internalising classifiers for monomorphisms}
\label{sec:intern-class-monom}
As we remarked at the end of the last section, h-propositions in \(\cubset\) are not necessarily the same as monomorphisms.
However, we have a third class we can consider: monomorphisms of the form \(\Delta(m)\) where \(m\) is a monomorphism in sets.
We might wonder if there are examples of h-propositions in \(\cubset\) that are monomorphisms but do not arise from monomorphisms in \(\sets\) because they are not in the image of \(\Delta\).

The aim of this section is to show that in fact this is \emph{not} the case: once we restrict to h-propositions that are monomorphisms they are necessarily of the form \(\Delta(m)\) where \(m\) is a monomorphism in \(\sets\).
Moreover, we will see a useful theorem that can be applied to classes of monomorphisms in \(\sets\) that are all classified by a single ``universal'' element \(m\).
For a monic fibration \(f : X \to Y\), once we know \(\Gamma(f)\) belongs to the class, we can deduce that \(f\) is a pullback of \(\Delta(m)\).

For this section and the next we only need a smaller subset of the properties of cubical sets.
In fact the results of this section will hold for any category of internal presheaves on an internal category \(\mathcal{C}\) in a locally cartesian closed category equipped with a weak factorisation system with the following properties, referring to the left class of the weak factorisation system as trivial cofibrations and the right class as fibrations.
\begin{enumerate}
  \item \(\mathcal{C}\) has a terminal object.
  \item Every object \(c\) of \(\mathcal{C}\) admits a map \(1 \to c\).
  \item Every map \(1 \to \yoneda c\) is a trivial cofibration.
\end{enumerate}

\begin{prop}
  Let \(m : P \rightarrowtail Q\) be a monomorphism in a locally cartesian closed category \(\mathcal{E}\). The following are equivalent:
  \begin{enumerate}
  \item In the internal language of \(\mathcal{E}\) we have \(\forallq{q, q' : Q}{(P_{q} \leftrightarrow P_{q'}) \rightarrow q = q'}\)
  \item For any monomorphism \(f : X \to Y\), there is at most one map \(\chi : Y \to Q\) forming the bottom map in a pullback diagram of the form below:
    \begin{equation}
      \label{eq:1}
      \begin{tikzcd}
        X \ar[r] \ar[d, "f"] \pbcorner & P \ar[d, "m"] \\
        Y \ar[r, "\chi"] & Q
      \end{tikzcd}
    \end{equation}
  \end{enumerate}
\end{prop}

\begin{proof}
  We first show the direction \((\Rightarrow)\). Suppose that we have two maps \(\chi, \chi' : Y \to Q\) fitting into the bottom map of a pullback as in \eqref{eq:1}. To show \(\chi = \chi'\), it suffices to prove in the internal language that for all \(y : Y\) we have \(\chi(y) = \chi'(y)\). However, we have equivalences \(P_{\chi(y)} \leftrightarrow X_{y}\) and \(P_{\chi'(y)} \leftrightarrow X_{y}\) by expressing the fact that the squares are pullbacks in the internal language. Combining these gives an equivalence \(P_{\chi(y)} \leftrightarrow P_{\chi'(y)}\) and so we deduce \(\chi(y) = \chi'(y)\) by assumption.

  We now show the direction \((\Leftarrow)\). We construct in the internal language the type \(Y := \sum_{q, q' : Q} P_{q} \leftrightarrow P_{q'}\). We obtain a monomorphism \(f : X \to Y\) by pulling back \(m\) along the projection map \(\pi_0 : \sum_{q, q' : Q} P_{q} \leftrightarrow P_{q'} \;\to\; Q\). We define \(\chi := \pi_0\). We define \(\chi'\) by projecting out the second \(Q\) component, i.e. \(\chi' := \pi_0 \circ \pi_1\). Using the equivalence \(P_{q} \leftrightarrow P_{q'}\), we can show \(f\) is also the pullback of \(m\) along \(\chi'\). Hence \(\chi = \chi'\) by assumption, and so we have that \(q = q'\) whenever \(P_{q} \leftrightarrow P_{q'}\).
\end{proof}

\begin{defi}
  If \(m : P \to Q\) satisfies the equivalent conditions above, we say it is \emph{extensional}. 
\end{defi}

\noindent
We recall some standard facts and terminology about extensional monomorphisms.
\begin{defi}
  We refer to the unique map \(\chi\), when it exists, as the \emph{classifying map} for the monomorphism \(f\) in \eqref{eq:1}. We say \(m : P \to Q\) is the \emph{classifier} for the class of monomorphisms obtained by pulling it back along arbitrary maps.
\end{defi}

\begin{rem}
  A given class of monomorphisms has at most one classifier up to isomorphism.
\end{rem}

\begin{prop}
  If \(m : P \to Q\) is extensional, then \(P\) is subterminal, and terminal whenever \(\existsq{q : Q}{P_q}\) holds in the internal language.
\end{prop}

\begin{proof}
  Internally, we can think of \(P\) as \(\sum_{q : Q} P_q\). Given \(q, q'\) such that \(P_q\) and \(P_{q'}\) are both inhabited, we have \(P_q \Leftrightarrow \top \Leftrightarrow P_{q'}\) and so \(q = q'\) by extensionality. Since any two elements of \(P_q = P_{q'}\) are equal, we can deduce that \(\sum_{q : Q} P_q\) has at most one element.
\end{proof}

\begin{prop}
  \label{prop:extnmonoinpretopos}
  Suppose we have exact quotients. Then every monomorphism is a pullback of an extensional monomorphism.
\end{prop}

\begin{proof}
  Given a monomorphism \(m : P \to Q\), we define an equivalence relation on \(Q\) by setting \(q \sim q'\) whenever we have \(P_q \leftrightarrow P_{q'}\). We define \(Q'\) to be the quotient \(Q/{\sim}\). Given \(x : Q'\) we define a proposition \(P'_x\) as \(\prod_{q : Q} x = [q] \rightarrow P_q\). By exactness, we have that for \(q, q' : Q\), whenever \([q] = [q']\) we have \(P_q = P_{q'}\). It follows that for \(q : Q\), \(P_q \;\leftrightarrow\; \prod_{q' : Q} ([q] = [q'] \rightarrow P_{q'})\). From this it follows that \(m\) is the pullback of \(m' := \sum_{x : Q'} P'_x\) along the quotient map.
\end{proof}

\noindent
When working with extensional monomorphisms \(P \to Q\) in the internal language, we will also write the fibre \(P_x\) as \([x]\) for \(x : Q\).

We now see the first key theorem, which shows that given an extensional monomorphism in our metatheory, the class of monomorphisms classified by the extensional monomorphism is essentially unchanged by passing into cubical sets via \(\Delta\).

\begin{lem}
  \label{lem:natispbpoints}
  Suppose that \(f : X \to Y\) is both a monomorphism and a fibration. For any map \(s : 1 \to [n]\) in the cube category \(\square\), the naturality square below is a pullback.
  \begin{displaymath}
    \begin{tikzcd}
      X_{n} \ar[d, tail, "f_n"] \ar[r, "X_{s}"] & X_{0} \ar[d, tail, "f_0"] \\
      Y_{n} \ar[r, "Y_{s}"] & Y_{0}
    \end{tikzcd}
  \end{displaymath}
\end{lem}

\begin{proof}
  In any case we can construct the pullback \(s^\ast(f_0)\) as follows:
  \begin{displaymath}
    \begin{tikzcd}
      X_{n} \ar[dr, swap, tail, "f_n"] \ar[r, "X_{s}"] & s^\ast(X_{0}) \ar[r] \ar[d, tail, "s^\ast(f_0)"] \pbcorner & X_{0} \ar[d, tail, "f_0"] \\
      & Y_{n} \ar[r, "Y_{s}"] & Y_{0}
    \end{tikzcd}
  \end{displaymath}

  \noindent
  Since \(f_n\) and \(s^\ast(f_0)\) are both monomorphisms, and we already have a map \(X_{n} \to s^\ast(X_{0})\) over \(Y_{n}\), to show they are isomorphic it suffices to construct a map in the other direction \(s^\ast(X_{0}) \to X_{n}\) over \(Y_{n}\). Elements of \(s^\ast(X_{0})\) correspond precisely to commutative squares as in the solid lines below.
  \begin{displaymath}
    \begin{tikzcd}
      \yoneda [0] \ar[r] \ar[d, "\yoneda s"] & X \ar[d, "f"] \\
      \yoneda [n] \ar[r] \ar[ur, dotted] & Y
    \end{tikzcd}
  \end{displaymath}

  \noindent
  However, since \(\yoneda s\) is a trivial cofibration by Lemma \ref{lem:termrepristcof} and \(f\) is a fibration by assumption, we can choose a diagonal filler as in the dotted line above. This precisely gives us a choice of element of \(X_{n}\) in the fibre of \(y \in Y_{n}\). Putting these together gives the required map \(s^\ast(X_{0}) \to X_{n}\).
\end{proof}

\begin{lem}
  \label{lem:natispball}
  Suppose that \(f : X \to Y\) is both a monomorphism and a fibration. For any map \(s : [n] \to [m]\) in the cube category the corresponding naturality square is a pullback.
  \begin{displaymath}
    \begin{tikzcd}
      X_{m} \ar[d, "f_m"] \ar[r, "X_{s}"] & X_{n} \ar[d, "f_n"] \\
      Y_{m} \ar[r, "Y_{s}"] & Y_{n}
    \end{tikzcd}
  \end{displaymath}
\end{lem}

\begin{proof}
  By Proposition \ref{prop:cubeshavepoints} we have some map \(t : 1 \to [n]\). We can therefore extend the diagram above as follows.
  \begin{displaymath}
    \begin{tikzcd}
      X_{m} \ar[d, "f_m"] \ar[r, "X_{s}"] & X_{n} \ar[d, "f_n"] \ar[r, "X_{t}"] & X_{0} \ar[d, "f_0"] \\
      Y_{m} \ar[r, "Y_{s}"] & Y_{n} \ar[r, "Y_{t}"] & Y_{0}
    \end{tikzcd}
  \end{displaymath}

  \noindent
  By Lemma \ref{lem:natispbpoints} both the right hand square and the whole rectangle are pullbacks. Hence the left hand square is also a pullback.
\end{proof}

\begin{thm}
  \label{thm:internaliseclassifier}
  Suppose that \(f : X \to Y\) is a map in cubical sets that is both a fibration and a monomorphism, and that \(\Gamma(f) : \Gamma(X) \to \Gamma(Y)\) is a pullback of an extensional monomorphism \(g : P \to Q\). Then \(f\) is a pullback of \(\Delta(g) : \Delta(P) \to \Delta(Q)\).
\end{thm}

\begin{proof}
  First note that by Proposition \ref{prop:cubeshavepoints} and Lemma \ref{lem:natispbpoints}, for each \(n\) the map \(f_n : X_{n} \to Y_{n}\) is a pullback of \(f_0 : X_{0} \to Y_{0}\) along some map, and so a pullback of \(g\). By extensionality, this determines a unique map \(\chi_n : Y_{n} \to Q\). To show this gives a morphism of cubical sets, we need to check naturality, which amounts to the following commutative triangles for each \(s : [n] \to [m]\).
  \begin{displaymath}
    \begin{tikzcd}
      Y_m \ar[d, swap, "Y_{s}"] \ar[r, "\chi_m"] & Q \\
      Y_n \ar[ur, swap, "\chi_n"] &
    \end{tikzcd}
  \end{displaymath}

  \noindent
  However, note that \(\chi_n \circ Y_s\) is a classifying map for \(f_m\), by observing that both squares in the diagram below are pullbacks; the left hand square by Lemma ~\ref{lem:natispball} and the right hand square by the definition of \(\chi_n\).  Since \(\chi_m\) is also a classifying map for \(f_m\) we have \(\chi_n \circ Y_s = \chi_m\) by extensionality.
  \begin{displaymath}
    \begin{tikzcd}
      X_m \ar[r, "X_s"] \ar[d] & X_n \ar[r] \ar[d] & P \ar[d] \\
      Y_m \ar[r, "Y_s"] & Y_n \ar[r, "\chi_n"] & Q
    \end{tikzcd}
  \end{displaymath}

  \noindent
  Finally, since pullbacks are computed pointwise, it is clear from the definition of \(\chi_n\) that \(f\) is the pullback of \(\Delta(g)\) along \(\chi\).
\end{proof}

\begin{rem}
  \label{rmk:sprop1}
  Fibrations which are also monomorphisms can be understood syntactically by augmenting type theory with a universe of strict propositions \cite[Section 4.3]{defproofirr}. As a consequence of Theorem ~\ref{thm:internaliseclassifier}, when cubical sets are constructed in the internal language of a topos we can interpret the universe of strict propositions as \(\Delta(\Omega)\).
\end{rem}

\section{\texorpdfstring{\(\neg\neg\)-Stable}{Double negation stable} h-propositions}
\label{sec:negneg-stable-h}

We have seen so far that in general it is best to view h-propositions not as types with ``at most one point'' but rather as types with many points that are all joined together by paths. Next, we saw in Section~\ref{sec:intern-class-monom} that the subclass of monic fibrations, which are h-propositions that really do have at most one point, is well behaved and corresponds closely to monomorphisms in sets.

We now restrict to the subclass of \(\neg\neg\)-stable h-propositions. The main motivation for doing this is that, as we will see, \(\neg\neg\)-stable h-propositions are necessarily monomorphisms up to equivalence, allowing us to apply the results from Section~\ref{sec:intern-class-monom}.

In contrast to the class of all monic fibrations, we have a clear definition of which h-propositions are \(\neg\neg\)-stable inside type theory.
As a consequence of this, we can define classes both of \(\neg\neg\)-stable propositions in sets and of \(\neg\neg\)-stable h-propositions in cubical sets.
Hence we can compare these two classes, and we will see that in fact they are closely related.
In particular, given a classifier for \(\neg\neg\)-stable propositions in sets, say \(\Omega_{\neg\neg}\), we can obtain a classifier of \(\neg\neg\)-stable h-propositions in cubical sets simply as the constant cubical set \(\Delta(\Omega_{\neg\neg})\).

Although the class of \(\neg\neg\)-stable h-propositions is a somewhat restricted class compared to the class of all h-propositions, it suffices for some key constructions. For this paper these applications are constructing the Dedekind real numbers and defining extended Church's thesis. The latter is related to the fact that \(\neg\neg\)-stable propositions play an important role in realizability models, and e.g. appear frequently in \cite{vanoosten}.

Again, we will not need all of the properties of cubical sets. The results of this section will hold for any category of internal presheaves in a locally cartesian closed category equipped with a weak factorisation system, such that in addition to the properties from Section~\ref{sec:intern-class-monom} we have the following.
\begin{enumerate}
  \item Dependent products preserve fibrations.
  \item For any \(Y\), the unique map \(\bot \to Y\) is a fibration.
\end{enumerate}

\noindent
Throughout this section, for a given map \(X \to Y\) we write \(\neg X\) to mean the negation of \(X\) computed in the slice category over \(Y\), i.e. functions to \(\bot\) using the local exponential over \(Y\).

\begin{lem}
  \label{lem:negismono}
  For any map \(f : X \to Y\), the negation \(f' : \neg X \to Y\) is a monomorphism. If \(f\) is a fibration, then so is \(f'\).
\end{lem}

\begin{proof}
  The exponential \((-)^X\) in \(\cubset/Y\) preserves limits and so in particular preserves subterminals. Since \(\bot \to Y\) is subterminal in \(\cubset/Y\), so is \(\neg X = \bot^X\).

  To show that \(f'\) is a fibration, we recall that the list of basic facts about cubical sets in Section~\ref{sec:review-cubical-sets} included the facts that the initial object is fibrant and dependent products preserve fibrations. The map \(\bot \to Y\) is the pullback of \(0 \to 1\) along the unique map \(Y \to 1\), and so also a fibration. The local exponential can be constructed using dependent product and pullback, and so also preserves fibrations. These two together suffice to show \(f' : \neg X \to Y\) is a fibration.
\end{proof}

\begin{lem}
  \label{lem:negpointspointsneg}
  For any map \(f : X \to Y\), we have an isomorphism between \((\neg X)_0\) and \(\neg X_0\) as subobjects of \(Y_0\).
\end{lem}

\begin{proof}
  Since these are both subobjects of \(Y_0\), to show they are isomorphic, it suffices to show they are logically equivalent over \(Y_0\).

  We first construct the map \((\neg X)_0 \to \neg X_0\) over \(Y_0\). By the adjunction between products and exponentials in \(\sets/Y_0\), it suffices to construct a map \(X_0 \times (\neg X)_0 \to \bot\). However, this can be obtained by simply applying \(\Gamma\) to the evaluation map \(X \times (\neg X) \to \bot\).

  We now construct the map \(\neg X_0 \to (\neg X)_0\) over \(Y_0\). We can explicitly describe \(\neg X_0\) as the subobject of \(Y_{0}\) consisting of \(y \in Y_0\) such that the fibre \(f_0^{-1}(y)\) is empty. Using this, we need to find a (necessarily unique) global section of \(\neg X\) over \(y\). This is the same as constructing a map \(\ulcorner y \urcorner^\ast X \to \bot\), where \(\ulcorner y \urcorner : 1 \to Y\) corresponds to \(y \in Y_0\) under the Yoneda equivalence. That is, for each \(n \in \square\) we need to derive a contradiction from the existence of \(x \in X_n\) such that \(f_n(x) = Y_!(y)\), where \(!\) is the unique map \([n] \to [0]\). However, for any \(n\) we can find a map \(s : [0] \to [n]\) by Proposition~\ref{prop:cubeshavepoints}, and so given any such \(x\) produce an element \(X_s(x)\) of \(X_0\), which must lie in the fibre of \(y\) since \(! \circ s\) is necessarily the identity on \([0]\).
\end{proof}

\begin{thm}
  \label{thm:negnegstableclassifier}
  Suppose \(m : 1 \to \Omega_{\neg \neg}\) is a classifier for all \(\neg \neg\)-stable propositions. Then \(\Delta(m) : \Delta(1) \to \Delta(\Omega_{\neg \neg})\) is a homotopy classifier for all \(\neg \neg\)-stable h-propositions in cubical sets.
\end{thm}

\begin{proof}
  Suppose that \(f : X \to Y\) is a \(\neg \neg\)-stable h-proposition. This implies that \(f\) is equivalent to its double negation, which we will write as \(f' : \neg \neg X \to Y\). By Lemma ~\ref{lem:negismono} \(f'\) is a monomorphism. By Lemma ~\ref{lem:negpointspointsneg} \(f'_0 : (\neg \neg X)_0 \to Y_0\) is equivalent to \(\neg (\neg X)_0\) and so also \(\neg \neg\)-stable. Hence \(f'_0\) is a pullback of \(m : 1 \to \Omega_{\neg \neg}\) and so by Theorem ~\ref{thm:internaliseclassifier} \(f'\) is a pullback of \(\Delta(m) : \Delta(1) \to \Delta(\Omega_{\neg \neg})\). Since \(f\) is equivalent to \(f'\) it is therefore a homotopy pullback of \(\Delta(m) : \Delta(1) \to \Delta(\Omega_{\neg \neg})\).

  Finally, to show \(\Delta(1) \to \Delta(\Omega_{\neg \neg})\) classifies the \(\neg \neg\)-stable h-propositions exactly, we need to verify that it is \(\neg \neg\)-stable itself. However, this follows from the fact that \(\Delta\) preserves initial object and dependent products, and so preserves double negation.
\end{proof}

\noindent
There are many situations where we have access to classifiers for \(\neg \neg\)-stable propositions, the most important being the following.

\begin{exa}
  Any topos has a classifier for all \(\neg \neg\)-stable propositions by defining it as the obvious subobject of the subobject classifier.
\end{exa}

\begin{exa}
  A category of assemblies has a classifier for \(\neg \neg\)-stable propositions, assuming they were constructed in a metatheory with the same. Namely, if \(\Omega_{\neg \neg}\) is a classifier for \(\neg \neg\)-stable propositions in sets, we define an assembly with underlying set \(\Omega_{\neg \neg}\) and uniform realizability predicate \(E(p) := \{0\}\). Every \(\neg \neg\)-stable monomorphism is a uniform map in the sense of \cite[Section 3.4]{vanoosten}. In this way, we can think of \(\neg \neg\)-stable h-propositions in cubical assemblies as proof irrelevant on two different levels. By Lemma ~\ref{lem:negismono} they are monomorphisms, and so types where any two elements are strictly equal. However, in addition their underlying monomorphism in assemblies is also \(\neg \neg\)-stable by Lemma ~\ref{lem:negpointspointsneg}, and so uniform, which can be seen as a form of proof irrelevance inherent to categories of assemblies and other realizability models.
\end{exa}

\begin{rem}
  To follow up on Remark ~\ref{rmk:sprop1}, as an alternative to interpreting strict propositions as the collection of all monomorphisms, we could instead restrict to only \(\neg \neg\)-stable propositions, and in particular interpret the universe of strict propositions as \(\Delta(\Omega_{\neg \neg})\).
\end{rem}

\section{The Dedekind real numbers}
\label{sec:two-defin-dedek}

We recall, e.g. from \cite[Section 3.6]{aczelrathjen} that the Dedekind reals can be defined constructively using the notion of \emph{left cut}, as given below.

\begin{defi}
  A \emph{Dedekind left cut} is a set \(L \subseteq \QQ\) satisfying the following properties:
  \begin{enumerate}
  \item (Boundedness) There exist rational numbers \(a \in L\) and \(b \notin L\).
  \item (Openness) For all \(a \in L\) there merely exists \(b \in L\) such that \(b > a\).
  \item (Locatedness) For all \(a < b \in \QQ\) either \(a \in L\) or \(b \notin L\).
  \end{enumerate}
\end{defi}

\begin{rem}
  \label{rmk:locatedimpliesdownclosed}
  It follows from locatedness that \(L\) is downwards closed.
\end{rem}

\noindent
In this paper we will, however, not use left cuts directly, but instead a variant that we refer to as \emph{cocut}. The reason for this is that it will turn out that cocuts are \(\neg \neg\)-stable as subsets of \(\QQ\), allowing us to apply the results of Section~\ref{sec:negneg-stable-h}.

\begin{defi}
  A \emph{co-left cut} or just a \emph{cocut} is a set \(C \subseteq \QQ\) satisfying the following properties:
  \begin{enumerate}
  \item (Boundedness) There exist rational numbers \(a \notin C\) and \(b \in C\).
  \item (Closedness) For all \(a \in \QQ\), if \(b \in C\) for all \(b > a\), then \(a \in C\).
  \item (Locatedness) For all \(a < b \in \QQ\) either \(a \notin C\) or \(b \in C\).
  \end{enumerate}
\end{defi}
\noindent
As before, note that any cocut is upwards closed by locatedness.

We can translate between the two definitions using the following operations. For each \(L \subseteq \QQ\) we define \(\neg L\) and \(L^<\) as follows:
\begin{align*}
  \neg L &:= \QQ \setminus L \\
  L^< &:= \{ a \in \QQ \;|\; \existsq{b \in L}{a < b} \}
\end{align*}

\begin{prop}
  \label{prop:cutcocut}
  \mbox{}
  \begin{enumerate}
  \item If \(L \subseteq \QQ\) is a left cut then \(\neg L\) is a cocut, and \(L = (\neg \neg L)^<\).
  \item If \(C \subseteq \QQ\) is a cocut then \((\neg C)^<\) is a left cut and \(C = \neg ((\neg C)^<)\).
  \end{enumerate}
\end{prop}

\begin{proof}
  Suppose first that \(L\) is a left cut. It is clear that \(\neg L\) is bounded, and closedness and locatedness of \(\neg L\) easily follow from openness and locatedness of \(L\) respectively. Given \(a \in L\), there exists \(b \in L\) such that \(a < b\). We then have \(b \in \neg \neg L\), and so \(a \in (\neg \neg L)^<\). Hence \(L \subseteq (\neg \neg L)^<\). Given \(a \in (\neg \neg L)^<\), there exists \(b \in \neg \neg L\) such that \(a < b\). By locatedness, either \(a \in L\) or \(b \notin L\). However, the latter contradicts \(b \in \neg \neg L\), and so we have \(a \in L\), giving \((\neg \neg L)^< \subseteq L\).

  Now suppose that \(C\) is a cocut. Note that \((\neg C)^<\) is open by definition, and it is bounded by the boundedness of \(C\). To check locatedness, suppose we are given \(a < b \in \QQ\). By locatedness of \(C\) we know that either \(\frac{a + b}{2} \notin C\) or \(b \in C\). The former implies \(a \in (\neg C)^<\) and the latter implies \(b \notin (\neg C)^<\).

  To check that \(C \subseteq \neg((\neg C)^<)\), suppose \(a \in C\). To show \(a \in \neg((\neg C)^<)\), we need to derive a contradiction from the assumption \(a \in (\neg C)^<\), so suppose there is \(b > a\) such that \(b \notin C\). However, since \(C\) is upwards closed, this contradicts \(a \in C\), as required. We now check that \(\neg((\neg C)^<) \subseteq C\). Suppose that \(a \in \neg((\neg C)^<)\). To show \(a \in C\), it suffices, by closedness, to check that for all \(b > a\), \(b \in C\). For any \(b > a\), we have by locatedness that either \(\frac{a + b}{2} \notin C\) or \(b \in C\). The former implies that \(a \in (\neg C)^<\), contradicting \(a \in \neg((\neg C)^<)\), and so we must have \(b \in C\), as required.
\end{proof}

\begin{prop}
  \label{prop:closednessaltdef}
  Let \(C\) be a cocut. Then \(a \in C\) if and only if for all \(n \in \NN\), \(a + \frac{1}{n + 1} \in C\).
\end{prop}

\begin{proof}
  The implication \((\Rightarrow)\) follows from the fact that \(C\) is upwards closed.

  It remains to check the implication \((\Leftarrow)\). Suppose that for all \(n \in \NN\), \(a + \frac{1}{n + 1} \in C\). To show \(a \in C\) it suffices by closedness to show that \(b \in C\) for all \(b > a\). For any \(b > a\) we can find \(n\) such that \(a < a + \frac{1}{n + 1} < b\). By assumption \(a + \frac{1}{n + 1} \in C\), and so \(b \in C\) since \(C\) is upwards closed.
\end{proof}

\begin{cor}
  If we have a classifier for \(\neg \neg\)-stable propositions in the metatheory, then there is a type of all Dedekind real numbers in cubical sets.
\end{cor}

\begin{proof}
  By Proposition~\ref{prop:cutcocut}, every cocut \(C\) is equivalent to \(\neg ((\neg C)^<)\) and so \(\neg \neg\)-stable. Hence we can construct the collection of all cocuts using the classifier for \(\neg \neg\)-stable h-propositions given in Theorem~\ref{thm:negnegstableclassifier}.
\end{proof}

\section{A model of extended Church's thesis}
\label{sec:model-extend-churchs}

\newcommand{\partfn}{\operatorname{Partial}_{\neg \neg}}

As stated in the conclusion to \cite{uemuraswan}, the main barrier to finding a model of extended Church's thesis was finding a good way to formulate partial functions within cubical sets. However, \(\neg\neg\)-stable propositions suffice for stating a formulation of extended Church's thesis for partial functions based on the axiom \(\mathbf{ECT}'_0\) appearing in \cite[Chapter 4, Section 5.5]{troelstravandalen}. Namely, given a classifier for \(\neg\neg\)-stable h-propositions, \(\Omega_{\neg \neg}\), we define for types \(X\) and \(Y\) the type \(\partfn(X, Y) := \sum_{D : X \to \Omega_{\neg\neg}} \prod_{x : X} [D(x)] \to Y\).

We will state Church's thesis using Kleene's \(T\) predicate and extraction function \(U\). Recall that \(T(e, x, z)\) is a primitive recursive predicate stating that \(z\) encodes a valid sequence of states for the \(e\)th Turing machine with input \(x\) starting with the initial state and ending with the halting state. \(U(z)\) is then the resulting output of the halting computation.

\begin{defi}
  \emph{Extended Church's thesis}, or \(\mathbf{ECT}\) is the axiom
  \[
    \prod_{f : \partfn(\NN, \NN)}\left\| \sum_{e : \NN}\prod_{x : \NN} \prod_{w : [\pi_0(f)(x)]} \sum_{z : \NN}T(e, x, z) \times U(z) =
      \pi_1(f)(x, w)\right\|
  \]
\end{defi}

\begin{rem}
  Note that \(\mathbf{ECT}\) does not assert the existence of a computable partial function with the same domain as \(f\) but rather with a domain which is a superset of that of \(f\), and in many cases the domain will be strictly larger. For example, define \(T \subseteq \NN\) to be the set of numbers \(e\) such that the computable function \(\varphi_e\) is total. In the presence of Markov's principle \(T\) is \(\neg \neg\)-stable, and so \(\mathbf{ECT}\) tells us that any function \(\NN^\NN \to \NN\) can be represented as a partial function \(\varphi_e : \NN \rightharpoondown \NN\) whose domain includes \(T\). However, the domain of any computable partial function is computably enumerable, whereas \(T\) is not computably enumerable, and so the domain of \(\varphi_e\) cannot be equal to \(T\).
\end{rem}

\noindent
First note that we have an ``absoluteness'' result for partial functions from \(\NN\) to \(\NN\) with \(\neg\neg\)-stable domain, i.e. the following proposition.
\begin{prop}
  \label{prop:partisdiscrete}
  The type of partial functions from \(\NN\) to \(\NN\) with \(\neg\neg\)-stable domain in homotopy type theory is implemented in cubical sets as \(\Delta(\partfn(\NN, \NN))\).
\end{prop}

\begin{proof}
  \(\Delta\) preserves all dependent products and sums, the natural number object, and by Theorem \ref{thm:negnegstableclassifier} also preserves the classifier for \(\neg\neg\)-stable propositions. But these suffice to construct \(\partfn(\NN, \NN)\).
\end{proof}

\begin{thm}
  The following axioms can be consistently added to Martin-L\"{o}f type theory:
  \begin{enumerate}
  \item Propositional truncation
  \item The axiom of univalence
  \item The existence of a classifier for \(\neg\neg\)-stable h-propositions
  \item Extended Church's thesis
  \item Markov's principle
  \end{enumerate}
\end{thm}

\begin{proof}
  Following \cite{uemuraswan} we first construct the cubical assemblies model of homotopy type theory by defining cubical sets internally in assemblies over the first Kleene algebra. We then define a reflective subuniverse where extended Church's thesis is forced to hold by nullification. Namely we nullify the family of propositions defined as the interpretation of the following types in cubical assemblies.
  \[
    f : \partfn(\NN, \NN) \vdash \left\| \sum_{e : \NN}\prod_{x : \NN} \prod_{w : [\pi_0(f)(x)]} \sum_{z : \NN}T(e, x, z) \times U(z) =
      \pi_1(f)(x, w)\right\|
  \]
  To ease notation, we define \(A\) and \(B\) as follows.
  \begin{align*}
    A &:= \partfn(\NN, \NN) \\
    B &:= \sum_{e : \NN} \prod_{x : \NN} \prod_{w : [\pi_0(f)(x)]} \sum_{z : \NN}T(e, x, z) \times U(z) = \pi_1(f)(x, w)
  \end{align*}

  \noindent
  By Proposition \ref{prop:partisdiscrete} the interpretation of \(A\) in cubical assemblies is discrete, and moreover is the image under \(\Delta\) of the interpretation of the same type in assemblies. Since the category of assemblies satisfies extended Church's thesis, by a similar argument to that in \cite[Corollary 3.1.4]{vanoosten}, the interpretation of \(A \vdash B\) in assemblies is well supported. We can therefore apply the same arguments as in \cite[Section 5.1]{uemuraswan} to show that the reflective subuniverse has the same natural number type and empty type as the original cubical assemblies model. The latter implies that the model is non trivial, and that every \(\neg\neg\)-stable h-proposition in the reflective subuniverse is already \(\neg\neg\)-stable in the original model. It follows that \(1 \to \Delta(\Omega_{\neg \neg})\) still acts as a classifier for \(\neg\neg\)-stable h-propositions in the reflective subuniverse. We can therefore use the same argument as in \cite[Section 6]{uemuraswan} to show that the resulting model is non trivial and satisfies extended Church's thesis and Markov's principle.
\end{proof}

\section{Weakly \texorpdfstring{\(\Pi^0_1\)}{Pi01} h-propositions}
\label{sec:weakly-pi0_1-h}

In Section \ref{sec:two-defin-dedek} we gave a construction of the Dedekind reals in cubical sets that relied on having a classifier for \(\neg\neg\)-stable propositions in our metatheory. Since this involves some impredicativity, it is not always viewed as constructively acceptable. We therefore also give a predicative proof using a smaller class of h-propositions that suffice to construct the Dedekind real numbers.

\begin{defi}
  \label{def:pi01mono}
  A monomorphism \(A \to B\) is \(\Pi^0_1\) if there is a function \(g : B \times \NN \to 2\) such that \(\prod_{b : B} (A_b \leftrightarrow \prod_{n : \NN} g(b, n) = 0)\).
\end{defi}

\begin{exa}
  Every exact locally cartesian closed category with natural number object has an extensional monomorphism \(1 \to \Omega_{\Pi^0_1}\) such that every \(\Pi^0_1\) monomorphism is a pullback of \(1 \to \Omega_{\Pi^0_1}\), as a special case of Proposition ~\ref{prop:extnmonoinpretopos}.
\end{exa}

\begin{exa}
  Categories of assemblies have classifiers for \(\Pi^0_1\)-monomorphisms, assuming they are constructed in a metatheory that also has a classifier for \(\Pi^0_1\)-monomorphisms.
\end{exa}

\begin{defi}
  An h-proposition \(f : X \to Y\) is \emph{weakly \(\Pi^0_1\)} if there is an h-proposition \(R \to Y \times \NN\) together with terms witnessing \(\prod_{y : Y} \prod_{n : \NN} \| R_{y, n} + \neg X_{y}\|\) and \(\prod_{y : Y} (X_y \leftrightarrow \prod_{n : \NN} R_{y, n})\).
\end{defi}

\noindent
The following two propositions are not formally required, but explain our choice of terminology.
\begin{prop}
  Every \(\Pi^0_1\) h-proposition is weakly \(\Pi^0_1\).
\end{prop}

\begin{proof}
  Suppose that we have \(g : B \times \NN \to 2\) as in Definition \ref{def:pi01mono}. Define \(R_{b, n} := g(b, n) = 0\). For any \(b : B\) and \(n : \NN\), either \(g(b, n) = 0\) or \(g(b, n) = 1\). The former is precisely \(R_{b, n}\), and the latter implies \(\neg \prod_{n : \NN} R_{b, n}\) and thereby \(\neg [b]\), and so we have \(R_{b, n} \vee \neg [b]\).
\end{proof}

\begin{prop}
  Suppose that every function \(\NN \to \NN\) is computable. Then a subobject of \(\NN^k\), say \(A \hookrightarrow \NN^k\) is a \(\Pi^0_1\)-monomorphism if and only if there is a primitive recursive formula \(\phi(x_1,\ldots,x_k; y)\) in the language of first order arithmetic such that
  \[
    \forallq {x_1,\ldots,x_k}{A(x_1,\ldots,x_k)} \leftrightarrow \forallq {y}{\phi(x_1,\ldots,x_k;y)}
  \]
\end{prop}

\begin{proof}
  Let \(g : \NN^k \times \NN \to 2\) be as in Definition \ref{def:pi01mono}. Let \(e\) be a code for a Turing machine whose output matches \(g\), i.e. for all \(x_1,\ldots,x_k, y\) we have \(\varphi_e(x_1,\ldots,x_k, y) = g(x_1,\ldots,x_k, y)\). By standard arguments we may assume we are given a primitive recursive bijection \(i : \NN \stackrel{\cong}{\to} \NN \times \NN\). We take \(\phi(x_1,\ldots,x_k; y)\) to be the formula stating that if \(\varphi_e(x_1,\ldots,x_k, \pi_0(i(y)))\) halts within \(\pi_1(i(y))\) steps then \(\varphi_e(x_1,\ldots,x_k, \pi_0(i(y))) = 0\), which is clearly primitive recursive.

  The converse is clear.
\end{proof}

\noindent
The motivation for the definition of weakly \(\Pi^0_1\) h-proposition is that we can apply it to the definition of the Dedekind reals in terms of cocuts, while also using some of our earlier observations about \(\neg \neg\)-stable h-propositions.
\begin{lem}
  \label{lem:cocutswklypi01}
  Every cocut \(C\) is weakly \(\Pi^0_1\).
\end{lem}

\begin{proof}
  We define \(R_{a, n} := a + \frac{1}{n + 1} \in C\). Locatedness tells us that for all \(n\), \(a \notin C\) or \(R_{a, n}\). Proposition ~\ref{prop:closednessaltdef} tells us that \(a \in C\) if and only if \(\prod_{n : \NN} R_{a, n}\).
\end{proof}

\begin{lem}
  \label{lem:pi01nnstable}
  Every weakly \(\Pi^0_1\) h-proposition is \(\neg \neg\)-stable.
\end{lem}

\begin{proof}
  We work internally in homotopy type theory. We assume we are given an element of \(\neg \neg X_y\) for some \(y : Y\). To show \(X_y\) we can equivalently prove \(\prod_{n : \NN} R_{y, n}\). For any \(n : \NN\) we have by assumption either \(R_{y, n}\) or \(\neg X_{y}\). The latter contradicts \(\neg \neg X_{y}\), and so we have \(R_{y, n}\). Since this is true for all \(n\), we deduce \(X_{y}\).
\end{proof}

\begin{thm}
  \label{thm:pi01hclassifier}
  If \(1 \to \Omega_{\Pi^0_1}\) is a classifier for \(\Pi^0_1\) monomorphisms in our metatheory, then every weakly \(\Pi^0_1\) h-proposition, \(f : X \to Y\), is a homotopy pullback of \(\Delta(1) \to \Delta(\Omega_{\Pi^0_1})\) in cubical sets.
\end{thm}

\begin{proof}
  We need to check that every weakly \(\Pi^0_1\) h-proposition \(f : X \to Y\) is equivalent to one obtained by pulling back \(\Delta(1) \to \Delta(\Omega_{\Pi^0_1})\). First note that we may assume without loss of generality that \(f\) is a monomorphism, since by Lemma ~\ref{lem:pi01nnstable} it is equivalent to the double negation \(\sum_{y : Y} \neg \neg X_y \to Y\), which is a monomorphism by Lemma ~\ref{lem:negismono}. In order to apply Theorem ~\ref{thm:internaliseclassifier} we need to check that \(\Gamma(X) \to \Gamma(Y)\) is \(\Pi^0_1\). By applying Corollary ~\ref{cor:detruncate} with \(X := \sum_{y : Y, n : \NN} R_{y, n} + \neg X_{y}\) and \(Y := Y \times \NN\), together with the definition of weakly \(\Pi^0_1\) h-proposition we have a section of \(\Gamma(R + \neg X) \to \Gamma (Y \times \NN)\). Since \(\Gamma\) preserves all limits and colimits and \(\NN\) is discrete, this gives us a map \(g : \Gamma (Y) \times \NN \to \Gamma(R) + \Gamma(\neg X)\). For each \(y \in \Gamma(Y)\) we can define a function \(g'_y : \NN \to 2\) where \(g'_y(n) = 0\) when \(g(y, n) = \inl(z)\) for \(z \in \Gamma(R)\) and \(g'_y(n) = 1\) when \(g(y, n) = \inr(\ast)\). Using the term witnessing \(\prod_{y : Y} (X_y \leftrightarrow \prod_{n : \NN} R_{y, n})\) we can show that each fibre \(\Gamma(X)_y\) is inhabited if and only if \(g'_y(n) = 0\) for all \(n\). Hence \(\Gamma(X) \to \Gamma(Y)\) is indeed \(\Pi^0_1\) and so we can apply Theorem ~\ref{thm:internaliseclassifier}.
\end{proof}

\begin{rem}
  Since we were able to explicitly define binary sequences \(g'_y\) in the proof above, it might appear at first that we did not need the extensionality condition and could have used instead e.g. the map \(1 \to 2^\NN\) pointing to the constantly zero sequence in place of the classifier \(1 \to \Omega_{\Pi^0_1}\). However, this would not work. The sequence \(g'_y\) depends on the choice of point \(y\), and so we could have different choices of sequence for each of two points joined by a path, whereas in order to get a well defined map to \(\Delta(\Omega_{\Pi^0_1})\) we need to assign the same element of \(\Omega_{\Pi^0_1}\) to both points. Note that when we defined such a map in Theorem ~\ref{thm:internaliseclassifier} we made essential use of extensionality.
\end{rem}

\begin{cor}
  Suppose we are given a classifier \(1 \to \Omega_{\Pi^0_1}\) for \(\Pi^0_1\) monomorphisms in our metatheory. Let \(U_n\) be a universe of small types. Then it holds in the interpretation of HoTT in cubical sets that every weakly \(\Pi^0_1\) h-proposition in \(U_n\) is equivalent to one belonging to \(\Delta(\Omega_{\Pi^0_1})\).
\end{cor}

\begin{proof}
  We apply Theorem \ref{thm:pi01hclassifier} where \(Y\) is the type of all weakly \(\Pi^0_1\) h-propositions in \(U_n\) and \(X \to Y\) the projection map from inhabited weakly \(\Pi^0_1\) h-propositions in \(U_n\).
\end{proof}

\begin{cor}
  \label{cor:realsexistpredicative}
  Assume that there is a classifier for all \(\Pi^0_1\) monomorphisms in our metatheory. Then there is a collection of all Dedekind real numbers in cubical sets.
\end{cor}

\begin{proof}
  Internally in HoTT we can think of \(\Delta(1) \to \Delta(\Omega_{\Pi^0_1})\) as a family of h-propositions, which by Theorem \ref{thm:pi01hclassifier} includes all weakly \(\Pi^0_1\) h-propositions. We now work internally in HoTT, and define a subtype of \(\Delta(\Omega_{\Pi^0_1})^\QQ\) consisting of those \(C : \QQ \to \Delta(\Omega_{\Pi^0_1})\) which are cocuts. We need to check that it holds internally in HoTT that every cocut belongs to this collection. However, for every cocut \(C : \QQ \to \hprop\), and every rational \(a : \QQ\), \(C(a)\) is weakly \(\Pi^0_1\) by Lemma ~\ref{lem:cocutswklypi01}, and so \(C\) is indeed equal to one in this collection.
\end{proof}

\section{A remark on proof theoretic strength}
\label{sec:remark-cons-strength}

\newcommand{\mltt}{\mathbf{MLTT}}

In \cite{rathjenunivalence} Rathjen observes that since it is possible to define models of type theory with univalence in a constructive and predicative metatheory, the proof theoretic strength of type theory is unchanged by adding the univalence axiom. In particular, writing \(\mathbf{MLTT}^-\) for the theory obtained by removing \(W\)-types from Martin-L\"{o}f type theory, and \(\mathbf{UA}\) for the univalence axiom, the proof theoretic strength of \(\mathbf{MLTT}^- + \mathbf{UA}\) is the same as \(\mltt^-\) \cite[Corollary 7.2]{rathjenunivalence}. From Corollary ~\ref{cor:realsexistpredicative} we can see the same argument applies with the addition of the Dedekind reals. Namely, write \(\RR_\mathbf{D}\) for the axiom that the Dedekind reals exist (at the first universe level, say). We then have the following result.

\begin{cor}
  \(\mltt^- + \mathbf{UA} + \RR_\mathbf{D}\) has the same strength as \(\mltt^-\), which is the same as \(\mathbf{ATR}_0\). Its proof theoretic ordinal is \(\Gamma_0\).
\end{cor}

\noindent
In \cite[Section 11.3]{hottbook} an alternative definition of real number is given, based on the Cauchy reals, but using a higher inductive principle that ensures Cauchy completeness, which does not necessarily hold for the Cauchy real numbers in the absence of the axiom of countable choice. Write \(\RR_{\mathbf{HIT}}\) for the axiom that the HIT reals, as defined in loc. cit., exist (at the first universe level, say). Although it is likely \(\RR_{\mathbf{HIT}}\) can be constructed in cubical sets by the same methods as in \cite{chmhitsctt}, such a proof would require an infinitary inductive definition in the metatheory, which is not available in absolutely predicative systems such as \(\mltt^-\). This suggests the following conjecture.

\begin{conj}
  \(\mltt^- + \mathbf{UA} + \RR_{\mathbf{HIT}}\) has strictly greater proof theoretic strength than that of \(\mltt^-\).
\end{conj}

\noindent
Note that in the presence of countable choice, the Cauchy real numbers are already Cauchy complete, and therefore satisfy the higher inductive principle for the higher inductive Cauchy reals. However, countable choice easily holds in many models of extensional type theory with propositional truncation, e.g. the regular locally cartesian closed category of sets within \(\mathbf{CZF} + \{\mathbf{wInacc}(n) \;|\; n > 0\} + \mathbf{RDC}\), as listed in \cite[Theorem 7.1]{rathjenunivalence}. Hence, if the conjecture above is true, it would provide a natural example of an axiom which raises the consistency strength of \(\mltt^-\) when combined with the univalence axiom, while having no effect on the consistency strength of extensional type theory.

\section{Conclusion}
\label{sec:conclusion}

We can think of h-propositions that are double negation stable as those that are proof irrelevant in a strong sense. One way that this manifests is in the key idea we saw in Lemma \ref{lem:negismono}: in cubical sets they are interpreted as monomorphisms, i.e. types where any two elements are strictly equal. We can therefore think of them as possessing no nondegenerate paths or homotopies, even up to strict equality. In particular we can obtain a classifier from the constant cubical set on the corresponding classifier in our metatheory.

When we construct cubical sets inside a realizability model, such as assemblies, we can additionally say that double negation stable h-propositions carry no computational information, in the sense of uniform maps of assemblies.

Although the class of double negation stable h-propositions is rather restricted, we saw two places where they can play a useful role. By defining Dedekind real numbers in terms of cocuts, we ensured that all of the computational information associated to a real number is contained within the terms witnessing boundedness and locatedness, with the underlying subset of \(\QQ\) entirely proof irrelevant.

The second place we used double negation stable h-propositions was in our formulation of extended Church's thesis. The domain of a partial function \(\NN \rightharpoondown \NN\) is a function \(D : \NN \to \mathbf{hProp}\). We should expect the partial function to be a computable partial function when for each \(n : \NN\), \(D(n)\) carries no computational information beyond \(n\) itself, which we can ensure by requiring that it is \(\neg \neg\)-stable. We made this precise through realizability, and gave an example of a model of HoTT where extended Church's thesis holds.

\bibliographystyle{alphaurl}
\bibliography{mybib}{}

\end{document}